\documentclass[12pt]{amsart}
\usepackage{amsfonts,amssymb,amscd,amsmath,enumerate,color, xypic, fancybox, graphics}
\def\frk{\mathfrak}               

\def\Phi{{\frk N}}
\def\opn#1#2{\def#1{\operatorname{#2}}} 
\opn\chara{char} 
\opn\length{\ell} 
\opn\pd{pd} 
\opn\rk{rk}
\opn\projdim{proj\,dim} 
\opn\injdim{inj\,dim} 
\opn\rank{rank}
\opn\depth{depth} 
\opn\grade{grade} 
\opn\height{height}
\opn\embdim{emb\,dim} 
\opn\codim{codim}

\opn\Tr{Tr} 
\opn\bigrank{big\,rank}
\opn\superheight{superheight}
\opn\lcm{lcm}
\opn\trdeg{tr\,deg}
\opn\reg{reg} 
\opn\lreg{lreg} 
\opn\ini{in} 
\opn\lpd{lpd}
\opn\size{size}
\opn\mult{mult}
\opn\dist{dist}
\opn\cone{cone}
\opn\lex{lex}
\opn\rev{rev}
\opn\div{div} \opn\Div{Div} \opn\cl{cl} \opn\Cl{Cl}
\opn\Spec{Spec} \opn\Supp{Supp} \opn\supp{supp} \opn\Sing{Sing}
\opn\Ass{Ass} \opn\Min{Min}
\opn\Ann{Ann} \opn\Rad{Rad} \opn\Soc{Soc}
\opn\Syz{Syz} \opn\Im{Im} \opn\Ker{Ker} \opn\Coker{Coker}
\opn\Am{Am} \opn\Hom{Hom} \opn\Tor{Tor} \opn\Ext{Ext}
\opn\End{End} \opn\Aut{Aut} \opn\id{id} \opn\ini{in}

\opn\nat{nat}
\opn\pff{pf}
\opn\Pf{Pf} \opn\GL{GL} \opn\SL{SL} \opn\mod{mod} \opn\ord{ord}
\opn\Gin{Gin}
\opn\Hilb{Hilb}\opn\adeg{adeg}\opn\std{std}\opn\ip{infpt}
\opn\Pol{Pol}
\opn\sat{sat}
\opn\Var{Var}
\opn\Gen{Gen}

\opn\aff{aff} \opn\con{conv} \opn\relint{relint} \opn\st{st}
\opn\lk{lk} \opn\cn{cn} \opn\core{core} \opn\vol{vol}
\opn\link{link} \opn\star{star}
\opn\gr{gr}

\def\Pc{{\mathcal P}}
\def\Qc{{\mathcal Q}}

\def\pot#1#2{#1[\kern-0.28ex[#2]\kern-0.28ex]}

\opn\dirlim{\underrightarrow{\lim}}
\opn\inivlim{\underleftarrow{\lim}}
\let\to=\rightarrow

\def\Implies{\ifmmode\Longrightarrow \else
        \unskip${}\Longrightarrow{}$\ignorespaces\fi}
\def\implies{\ifmmode\Rightarrow \else
        \unskip${}\Rightarrow{}$\ignorespaces\fi}
\def\iff{\ifmmode\Longleftrightarrow \else
        \unskip${}\Longleftrightarrow{}$\ignorespaces\fi}

\let\:=\colon
\newtheorem{Theorem}{Theorem}[section]

\newtheorem{Corollary}[Theorem]{Corollary}
\newtheorem{Proposition}[Theorem]{Proposition}
\newtheorem{Remark}[Theorem]{Remark}

\newtheorem{Question}[Theorem]{Question}
\let\epsilon\varepsilon
\let\phi=\varphi
\let\kappa=\varkappa
\def\qed{\ifhmode\textqed\fi
      \ifmmode\ifinner\quad\qedsymbol\else\dispqed\fi\fi}
\def\textqed{\unskip\nobreak\penalty50
       \hskip2em\hbox{}\nobreak\hfil\qedsymbol
       \parfillskip=0pt \finalhyphendemerits=0}
\def\dispqed{\rlap{\qquad\qedsymbol}}

\opn\dis{dis}
\opn\height{height}
\opn\dist{dist}
\def\pnt{{\raise0.5mm\hbox{\large\bf.}}}

\opn\Lex{Lex}
\opn\conv{conv}

\begin{document}
\title{Non-Koszul quadratic Gorenstein toric rings}
\author[K.~Matsuda]{Kazunori Matsuda}
\address[Kazunori Matsuda]{Department of Pure and Applied Mathematics,
	Graduate School of Information Science and Technology,
	Osaka University,
	Suita, Osaka 565-0871, Japan}
\email{kaz-matsuda@ist.osaka-u.ac.jp}
\subjclass[2010]{Primary 13P10, Secondary 52B20}
\keywords{Koszul algebra, quadratic algebra, Gr\"{o}bner bases, stable set polytope. }
\begin{abstract}
Koszulness of Gorenstein quadratic algebras of small socle degree is studied. 
In this note, we construct non-Koszul Gorenstein quadratic 
toric ring such that its socle degree is more than 3 by using stable set polytopes.  
\end{abstract}
\maketitle

\section*{Introduction}

Let $K$ be a field and $S = K[x_1, \ldots, x_n]$ a polynomial ring over $K$. 
Let $R = S/I$ be a standard graded $K$-algebra with respect to the grading $\deg x_i = 1$ 
for all $1 \le i \le n$, where $I$ is a homogeneous ideal of $S$. 
Let $R_{+}$ denote the homogeneous maximal ideal of $R$. 
For an $R$-module $M$, we denote $\beta_{ij}^{R}(M)$ by the $(i, j)$-th graded betti number 
of $M$ as an $R$-module. 

The Koszul algebra was originally introduced by Priddy \cite{Priddy}. 
A standard graded $K$-algebra $R$ is said to be {\em Koszul} if 
the residue field $K = R/R_{+}$ has a linear $R$-free resolution as an $R$-module, 
that is, $\beta_{ij}^{R}(K) = 0$ if $i \neq j$.   
Since $\beta_{2j}^{R}(K) = 0$ for all $j > 2$, hence 
Koszul algebras are {\em quadratic}, where $R = S/I$ is said to be quadratic 
if $I$ is generated by homogeneous elements of degree 2. 
Every quadratic complete intersection is Koszul by Tate's theorem \cite{Tate}.  
Moreover, $R = S/I$ is Koszul if I has a quadratic Gr\"{o}bner bases 
by Fr\"{o}berg's theorem \cite{Fr} and the fact that $\beta_{ij}^{R}(K) \le \beta_{ij}^{R'}(K)$ 
for all $i, j$ and for all monomial order $<$ on $S$, where $R' = S/{\rm in}_{<} (I)$. 
The notion of Koszul algebra has played an important role in the research 
on graded $K$-algebras,  
and various Koszul-like algebras have been introduced, e.g., universally Koszul \cite{UKoszul}, strongly Koszul \cite{HHR}, initially Koszul \cite{IKoszul}, 
sequentially Koszul  \cite{seqKoszul}, etc. 

Koszulness of toric rings of integral convex polytopes is studied. 
Let $\Pc \subset \mathbb{R}^n$ be an integral convex polytope, i.e.,  a convex polytope 
each of whose vertices belongs to $\mathbb{Z}^n$, 
and let $\Pc \cap \mathbb{Z}^n = \{ {\bf a}_1, \ldots, {\bf a}_m \}$. 
Assume that $\mathbb{Z}{\bf a}_1 + \cdots + \mathbb{Z}{\bf a}_m = \mathbb{Z}^n$. 
Let $K[X^{\pm 1}, t] := K[x_1, x_1^{-1}, \ldots, x_n, x_n^{-1}, t]$ 
be the Laurent polynomial ring in $n + 1$ variables over $K$. 
Given an integer vector ${\bf a} = (a_1, \ldots, a_n) \in \mathbb{Z}^n$, we put 
$X^{\bf a}t = x_1^{a_1} \cdots x_n^{a_n}t \in K[X^{\pm 1}, t]$. 
The {\em toric ring} of $\Pc$, denoted by $K[\Pc]$, is the subalgebra of $K[X^{\pm 1}, t]$ 
generated by $\{ X^{{\bf a}_1}t, \ldots, X^{{\bf a}_m}t \}$ over $K$. 
Note that $K[\Pc]$ can be regarded as a standard graded $K$-algebra 
by setting $\deg X^{{\bf a}_i}t = 1$. 
The {\em toric ideal} $I_{\Pc}$ is the kernel of a surjective ring homomorphism 
$\pi : K[Y] = K[y_1, \ldots, y_m] \to K[\Pc]$ defined by $\pi (y_i) = X^{{\bf a}_i}t$ 
for $1 \le i \le m$. 
Then $K[\Pc] \cong K[Y] / I_{\Pc}$. 
It is known that $I_{\Pc}$ is generated by homogeneous binomials. 

Note that the following implications hold:

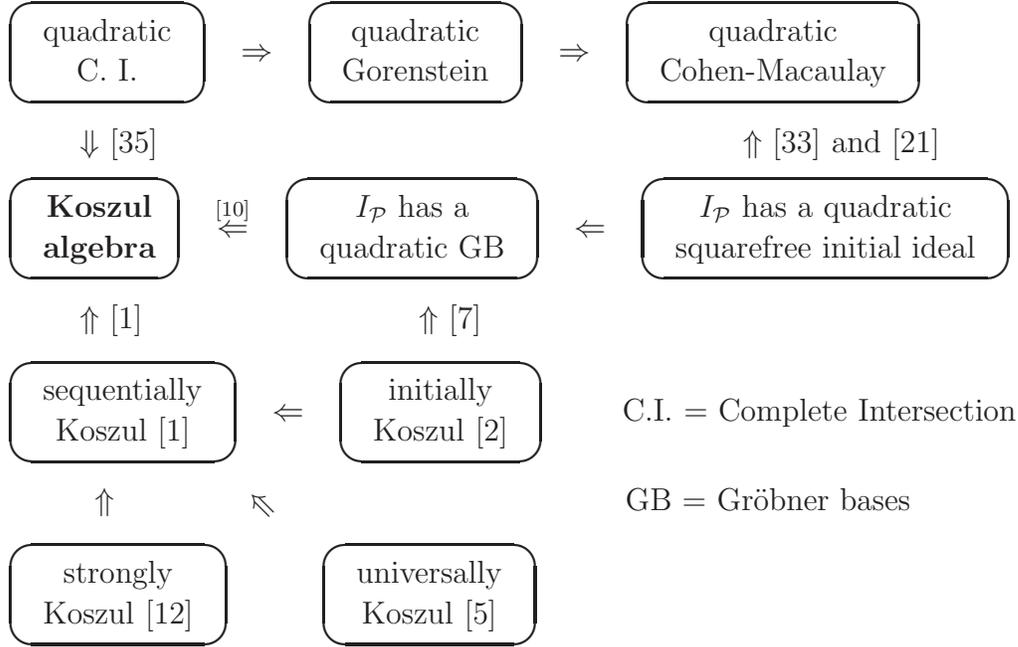
\vspace{15mm}

\Ovalbox{
\begin{tabular}{c}quadratic \\ C. I.  \end{tabular}
} \hspace{2mm} $\Rightarrow$ \hspace{2mm}
\Ovalbox{
\begin{tabular}{c}quadratic \\ Gorenstein \end{tabular}
} \hspace{2mm} $\Rightarrow$ \hspace{2mm}
\Ovalbox{
\begin{tabular}{c} quadratic \\ Cohen-Macaulay \end{tabular}
}  

\vspace{3mm}

\hspace{8mm} $\Downarrow$ \cite{Tate} \hspace{75mm} $\Uparrow$ \cite{Sturmfels} and \cite{Hochster}

\vspace{2mm}

 \Ovalbox{
 \begin{tabular}{c} {\bf Koszul}  \\ {\bf algebra} \end{tabular}} \hspace{2mm} 
 $\overset{[10]}{\Leftarrow}$ \hspace{2mm} 
\Ovalbox{
\begin{tabular}{c} $I_{\Pc}$ has a \\ quadratic GB \end{tabular}
} \hspace{2mm} $\Leftarrow$ \hspace{2mm}
\Ovalbox{
\begin{tabular}{c} $I_{\Pc}$ has a quadratic  \\ squarefree initial ideal \end{tabular}
} 

\vspace{3mm}

\hspace{8mm} $\Uparrow$  \cite{seqKoszul} \hspace{34mm} $\Uparrow$ \cite{CRV}

\vspace{3mm}

\Ovalbox{
\begin{tabular}{c} sequentially  \\ Koszul \cite{seqKoszul} \end{tabular}
} 
 \hspace{2mm} $\Leftarrow$ \hspace{2mm} 
\Ovalbox{
\begin{tabular}{c} initially  \\ Koszul \cite{IKoszul} \end{tabular}
}  \hspace{8mm} C.I. = Complete Intersection 
 
\vspace{3mm}
 
\hspace{10mm} $\Uparrow$ \hspace{15mm} \rotatebox{135}{$\Rightarrow$} 
\hspace{43.4mm} GB = Gr\"{o}bner bases

\vspace{3mm}
 
\Ovalbox{
\begin{tabular}{c} strongly  \\ Koszul \cite{HHR} \end{tabular}
}  \hspace{10mm}
\Ovalbox{
\begin{tabular}{c} universally \\ Koszul \cite{UKoszul} \end{tabular}
} 

\vspace{15mm} 

In addition, it is known the following: \\

\begin{enumerate}
	\item Conca-De Negri-Rossi posed a conjecture that the defining ideal of a strongly Koszul algebra has a quadratic Gr\"{o}bner bases \cite[Question 13 (1)]{CDR}. 
	This conjecture is true for the toric ring of edge polytope \cite{HMS}, 
	order polytope \cite{HHR}, stable set polytope \cite{Matsuda} and 
	cut polytope \cite{Shibata}. \\
	\item A squarefree strongly Koszul toric ring is compressed \cite[Theorem 2.1]{MO}, 
	where $K[\Pc]$ is said to be {\em compressed} if 
	$\sqrt{{\rm in}_{<} (I_{\Pc})} = {\rm in}_{<} (I_{\Pc})$ 
	for any reverse lexicographic order $<$ on $K[Y]$. 
	In particular, a squarefree strongly Koszul toric ring is quadratic Cohen-Macaulay. \\
	\item Many of toric rings associated with integral convex polytopes 
	whose toric ideals has a quadratic Gr\"{o}bner bases are constructed 
	(e.g., \cite{CFS}, \cite{Hibi1987}, \cite{HM}, \cite{HMOS}, \cite{GF}, \cite{HMT}).
	In other words, many of Koszul toric rings 
	associated with integral convex polytopes are constructed. \\
	\item Quadratic algebra is not always Koszul 
	(see \cite[Example 2.1]{notKoszul}, \cite[Example 3]{RS}). 
	Note that both of these examples are Cohen-Macaulay but are not Gorenstein. 
	
	On the other hand, Koszulness of Gorenstein quadratic algebras is studied. 
	For a standard graded $K$-algebra $R = \oplus_{i \ge 0} R_i$ with $\dim R = d$, 
	we denote by
	\[
	H_{R}(t) = \sum_{i \ge 0} \dim_{K} R_i t^i = \frac{h_0 + h_1 t + \cdots + h_s t^s}{(1 - t)^d}
	\]
	the {\em Hilbert series} of $R$, where $h_s \neq 0$, and 
	we say that $h(R) := (h_0, h_1, \ldots, h_s)$ is the {\em h-vector} of $R$ 
	and the index $s$ is the {\em socle degree} of $R$. 	
	It is known that $h_0 = 1$ and if $R$ is Gorenstein then $h_i = h_{s - i}$ 
	for all $0 \le i \le \lfloor s/2 \rfloor$ (\cite[Theorem 4.4]{Stanley}). 
	Conca-Rossi-Valla proved that if $R$ is a quadratic Gorenstein with 
	$h(R) = (1, n, 1)$ 
	(in this case $n \ge 2$ since $R$ is quadratic) 
	then $R$ is Koszul \cite[Proposition 2.12]{CRV}. 
	
	The case for $s = 2$ is also studied.  
	Let $R$ be a quadratic Gorenstein with $h(R) = (1, n, n, 1)$ 
	(in this case $n \ge 3$ since $R$ is quadratic). 
	If $n = 3$, then $R$ is quadratic complete intersection, hence $R$ is Koszul. 
	Conca-Rossi-Valla proved that $R$ is Koszul if $n = 4$ \cite[Theorem 6.15]{CRV} 
	and Caviglia proved that $R$ is Koszul if $n = 5$ in his unpublished master thesis. 
	The case for $n \ge 6$ is still open. \\
\end{enumerate}

In this note, we focus on (4). 
In Section 1, we remark about known result of toric rings and toric ideals of stable 
set polytopes, and construct non-Koszul quadratic Gorenstein toric rings 
by using stable set polytopes. 
In Section 2, we present some questions. 

\begin{Remark}
\normalfont
In this note, we use Macaulay2 \cite{Macaulay2} to check to be not Koszul.  
About checking of non-Koszulness by using Macaulay2, see \cite[p. 289]{Four}. 
\end{Remark} 

\section{Stable set polytope and \\non-Koszul quadratic Gorenstein toric ring}

The stable set polytope is an integral convex polytope associated with 
stable sets of a simple graph. 

Let $G$ be a finite simple graph on the vertex set $[n] = \{1, 2, \ldots, n\}$ 
and let $E(G)$ denote the set of edges of $G$. 
Recall that a finite graph is {\em simple} if it possesses no loops or multiple edges. 
We denote by $\overline{G}$ the complement graph of $G$. 

Given a subset $W \subset [n]$, we define the $(0, 1)$-vector 
$\rho(W) = \sum_{i \in W} {\bf e}_i \in \mathbb{R}^n$, 
where ${\bf e}_i$ is the $i$-th unit coordinate vector of $\mathbb{R}^n$. 
In particular, $\rho(\emptyset)$ is the origin of $\mathbb{R}^n$. 

A subset $W \subset [n]$ is said to be {\em stable} if $\{i, j\} \not\in E(G)$ 
for all $i, j \in W$ with $i \neq j$. 
Note that the empty set and each single-element subset of $[n]$ are stable. 
Let $S(G)$ denote the set of all stable sets of $G$. 
The {\em stable set polytope} 
of a simple graph $G$, denoted by $\Qc_{G}$, is the convex hull of 
$\{ \rho(W) \mid W \in S(G) \}$. 
By definition, $\Qc_{G}$ is a $(0, 1)$-polytope and 
$K[\Qc_{G}] = K[t \cdot \prod_{i \in W} x_i \mid W \in S(G)] \subset K[x_1, \ldots, x_n, t]$. 
Note that $\dim K[\Qc_{G}] = n + 1$. 
Let $K[Y] = K[y_{W} \mid W \in S(G)]$ be the polynomial ring over $K$. 
Now we define a surjective ring homomorphism 
$\pi : K[Y] \to K[\Qc_{G}]$ by $\pi(y_{W}) = t \cdot \prod_{i \in W} x_i$ 
and let $I_{\Qc_{G}} = \ker \pi$. 

To state known results of the toric ring $K[\Qc_{G}]$ and the toric ideal 
$I_{\Qc_{G}}$ of the stable set polytope $\Qc_{G}$ of a simple graph $G$, 
we introduce some classes of graphs. 
About terminologies for the graph theory, see \cite{D}. 

A {\em cycle} graph with length $n$, denoted by $C_n$, is a connected graph 
which satisfies $E(C_n) = \{ \{1, 2\}, \{2, 3\}, \ldots, \{n-1, n\}, \{1, n\} \}$. 
An {\em odd cycle} is a cycle such that its length is odd. 

A graph $G$ is said to be {\em perfect} if the chromatic number of every induced subgraph 
of $G$ is equal to the size of the largest clique of that subgraph. 
A graph $G$ is perfect if and only if both $G$ and $\overline{G}$ are $(C_{2n + 3}, n \ge 1)$-free 
\cite{CRST}. 

The {\em comparability} graph $G(P)$ of a partially ordered set $P = ([n], <_P)$ is 
the graph such that $V(G(P)) = [n]$ and $\{i, j\} \in E(G(P))$ if and only if $i <_P j$ or $j <_P i$. 
A graph $G$ is said to be {\em comparability} if $G$ is the comparability graph of some 
partially ordered set. 
Forbidden induced subgraphs of comparability graphs are known (see \cite[p.13]{Mancini}). 

A graph $G$ is said to be {\em bipartite} if there exist $V_1, V_2$ with $V_1 \cup V_2 = V(G)$ 
and $V_1 \cap V_2 = \emptyset$ such that if $\{i, j\} \in E(G)$ then either 
$i \in V_1$ and $j \in V_2$ or $i \in V_2$ and $j \in V_1$. 
It is known that a graph $G$ is bipartite if and only if $G$ is $(C_{2n + 1}, n \ge 1)$-free. 

A graph $G$ is said to be {\em almost bipartite} (see \cite[p.87]{EN}) if there exists a vertex $v$ 
such that the induced subgraph $G_{[n] \setminus v}$ is bipartite.

\begin{Remark}
\normalfont

It is known that \\
\begin{enumerate}
	\item Let $G$ be a perfect graph. 
	Then $K[\Qc_{G}]$ is Gorenstein if and only if all maximal cliques of $G$ 
	have the same cardinality \cite[Theorem 2.1(b)]{Gorenstein}. \\
	\item Let $G(P)$ be the comparability graph of a partially ordered set $P$. 
	Then $K[\Qc_{G(P)}]$ is Koszul since $\Qc_{G(P)}$ is equal to the chain polytope 
	of $P$ and the toric ideal of a chain polytope has 
	a squarefree quadratic initial ideal (see \cite[Corollary 3.1]{Chain}). \\
	\item If $G$ is almost bipartite, then $K[\Qc_{G}]$ is Koszul 
	since its toric ideal $I_{\Qc_{G}}$ has 
	a squarefree quadratic initial ideal (see \cite[Theorem 8.1]{EN}). \\
	\item Let $G$ be a graph such that $\overline{G}$ is bipartite. 
	Then $K[\Qc_{G}]$ is  quadratic if and only of it is Koszul \cite[Corollary 3.4]{MOS}. \\
\end{enumerate}
\end{Remark}

Hence, if $K[\Qc_{G}]$ is quadratic but not Koszul, 
then $G$ is neither comparability nor almost bipartite, and $\overline{G}$ is 
not bipartite. 
From this fact and the classifications of these graphs, we have

\begin{Proposition}
Let $G$ be a graph on $[n]$. 
If $K[\Qc_{G}]$ is non-Koszul quadratic Gorenstein, then $n \ge 7$. 
\end{Proposition} 
\begin{proof}
First, we assume that $n \le 5$. 
Then $G$ is a comparability graph if $G$ is not $C_5$. 
Since $C_5$ is almost bipartite, we have that $K[\Qc_{G}]$ is Koszul 
if $n \le 5$. 

Next, we assume that $n = 6$. 
If $G$ is not connected, then $G$ is a comparability graph if $G$ is not 
$C_5 \cup K_1$. 
Since $C_5 \cup K_1$ is almost bipartite, we have that $K[\Qc_{G(P)}]$ is Koszul. 

Assume that $G$ is connected. 
From the classifications of comparability and almost bipartite graphs, 
$G$ is one of the following (see \cite[p.10]{Matsuda}): \\

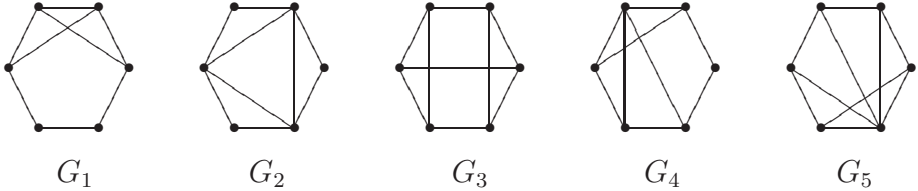
	\begin{xy}
		\ar@{} (0,  0); (10,  0) *\dir<4pt>{*} = "A1", 
		\ar@{-} "A1"; (6,  -8) *\dir<4pt>{*} = "B1", 
		\ar@{-} "B1"; (10,  -16) *\dir<4pt>{*} = "C1",
		\ar@{-} "C1"; (18,  -16) *\dir<4pt>{*} = "D1",
		\ar@{-} "D1"; (22,  -8) *\dir<4pt>{*} = "E1",
		\ar@{-} "E1"; (18,  0) *\dir<4pt>{*} = "F1",
		\ar@{-} "A1"; "F1",
		\ar@{-} "A1"; "E1",
		\ar@{-} "B1"; "F1",
		\ar@{} (0,  0); (36,  0) *\dir<4pt>{*} = "A2", 
		\ar@{-} "A2"; (32,  -8) *\dir<4pt>{*} = "B2", 
		\ar@{-} "B2"; (36,  -16) *\dir<4pt>{*} = "C2",
		\ar@{-} "C2"; (44,  -16) *\dir<4pt>{*} = "D2",
		\ar@{-} "D2"; (48,  -8) *\dir<4pt>{*} = "E2",
		\ar@{-} "E2"; (44,  0) *\dir<4pt>{*} = "F2",
		\ar@{-} "A2"; "F2",
		\ar@{-} "B2"; "D2",
		\ar@{-} "B2"; "F2",
		\ar@{-} "D2"; "F2",
		\ar@{} (0,  0); (62,  0) *\dir<4pt>{*} = "A3", 
		\ar@{-} "A3"; (58,  -8) *\dir<4pt>{*} = "B3", 
		\ar@{-} "B3"; (62,  -16) *\dir<4pt>{*} = "C3",
		\ar@{-} "C3"; (70,  -16) *\dir<4pt>{*} = "D3",
		\ar@{-} "D3"; (74,  -8) *\dir<4pt>{*} = "E3",
		\ar@{-} "E3"; (70,  0) *\dir<4pt>{*} = "F3",
		\ar@{-} "A3"; "F3",
		\ar@{-} "A3"; "C3",
		\ar@{-} "B3"; "E3",
		\ar@{-} "D3"; "F3",
		\ar@{} (0,  0); (88,  0) *\dir<4pt>{*} = "A4", 
		\ar@{-} "A4"; (84,  -8) *\dir<4pt>{*} = "B4", 
		\ar@{-} "B4"; (88,  -16) *\dir<4pt>{*} = "C4",
		\ar@{-} "C4"; (96,  -16) *\dir<4pt>{*} = "D4",
		\ar@{-} "D4"; (100,  -8) *\dir<4pt>{*} = "E4",
		\ar@{-} "E4"; (96,  0) *\dir<4pt>{*} = "F4",
		\ar@{-} "A4"; "F4",
		\ar@{-} "A4"; "C4",
		\ar@{-} "A4"; "D4",
		\ar@{-} "B4"; "F4",
		\ar@{} (0,  0); (114,  0) *\dir<4pt>{*} = "A5", 
		\ar@{-} "A5"; (110,  -8) *\dir<4pt>{*} = "B5", 
		\ar@{-} "B5"; (114,  -16) *\dir<4pt>{*} = "C5",
		\ar@{-} "C5"; (122,  -16) *\dir<4pt>{*} = "D5",
		\ar@{-} "D5"; (126,  -8) *\dir<4pt>{*} = "E5",
		\ar@{-} "E5"; (122,  0) *\dir<4pt>{*} = "F5",
		\ar@{-} "A5"; "F5",
		\ar@{-} "A5"; "D5",
		\ar@{-} "B5"; "D5",
		\ar@{-} "C5"; "E5",
		\ar@{-} "D5"; "F5",
	\end{xy}

\vspace{4mm}

\ \ \ \ \ \ $G_1$ \ \ \ \ \ \ \ \ \ \ \ \ \ \ $G_2$  \ \ \ \ \ \ \ \ \ \ \ \ \ \ \ $G_3$  
\ \ \ \ \ \ \ \ \ \ \ \ \ \ $G_4$  \ \ \ \ \ \ \ \ \ \ \ \ \ \ $G_5$\\

Then we can see that 

\begin{itemize}
	\item $K[\Qc_{G_1}]$ is not Gorenstein since $h(K[\Qc_{G_1}]) = (1, 7, 10, 3)$. 
	\item $K[\Qc_{G_2}]$ is Koszul since $I_{\Qc_{G_2}}$ has a quadratic Gr\"{o}bner bases. 
	\item $\overline{G_3}$ is $C_6$, hence bipartite. 
	\item $K[\Qc_{G_4}]$ is not Gorenstein since $h(K[\Qc_{G_4}]) = (1, 6, 8, 2)$. 
	\item $K[\Qc_{G_5}]$ is Koszul since $I_{\Qc_{G_5}} = I_{\Qc_{C_5}}$ and 
	$I_{\Qc_{C_5}}$ has a quadratic Gr\"{o}bner bases. 
\end{itemize}
Therefore we have the desired conclusion. 
\end{proof}
 
For each integer $k \ge 3$, the complement of a odd cycle $C_{2k + 1}$, denoted by 
$\overline{C_{2k + 1}}$, is neither comparability nor almost bipartite. 
Note that $\overline{C_{2k + 1}}$ is not perfect and 
$S(\overline{C_{2k + 1}}) = \{\emptyset, \{ 1 \}, \{ 2 \}, \ldots, \{ 2k + 1 \}, \{1, 2\}, \{2, 3\},\ldots, \{2k, 2k + 1\}, \{1, 2k + 1\}  \}$. 

Let $K[Y] = K[y_{\emptyset}, y_{\{1\}}, y_{\{2\}}, \ldots, y_{\{2k + 1\}}, y_{\{1, 2\}}, y_{\{2, 3\}}, \ldots, y_{\{2k, 2k + 1\}}, y_{\{1, 2k + 1\} }]$. 
Now we study the toric ring
\[ 
K[\Qc_{\overline{C_{2k + 1}}}] \cong \frac{K[Y]}{I_{\Qc_{\overline{C_{2k + 1}}}}}. 
\] 

$\ $

\begin{Proposition}
\normalfont
We have the following:
	\begin{enumerate}
		\item $K[\Qc_{\overline{C_{2k + 1}}}]$ is quadratic Cohen-Macaulay 
		for all $k \ge 3$. 
		\item $K[\Qc_{\overline{C_{2k + 1}}}]$ is not Gorenstein for all $k \ge 4$. 
		\item $K[\Qc_{\overline{C_{7}}}]$ is Gorenstein. 
		\item $I_{\Qc_{\overline{C_{2k + 1}}}}$ possesses no quadratic Gr\"{o}bner bases
		for all $k \ge 3$. 
	\end{enumerate}
\end{Proposition}
\begin{proof}
	\begin{enumerate}
		\item First, by \cite[Theorem 2.1]{MOS}, we have that 
		$K[\Qc_{\overline{C_{2k + 1}}}]$ is normal. 
		Hence $K[\Qc_{\overline{C_{2k + 1}}}]$ is Cohen-Macaulay. 
		Moreover, by \cite[Theorem 3.2]{MOS}, we have that 
		the toric ideal $I_{\Qc_{\overline{C_{2k + 1}}}}$ is generated by the following 
		$4k + 2$ binomials: \\
			\begin{itemize}
				\item $y_{\{i\}} y_{\{i + 1\}} - y_{\emptyset} y_{\{i, i + 1\}}$\ \ \  $(1 \le i \le 2k)$;
				\item $y_{\{1\}} y_{\{2k + 1\}} - y_{\emptyset} y_{\{1, 2k + 1\}}$; 
				\item $y_{\{i\}} y_{\{i + 1, i + 2\}} - y_{\{i + 2\}} y_{\{i, i + 1\}}$\ \ \  $(1 \le i \le 2k - 1)$;
				\item $y_{\{2k\}} y_{\{1, 2k + 1\}} - y_{\{1\}} y_{\{2k, 2k + 1\}}$, $y_{\{2k + 1\}} y_{\{1, 2\}} - y_{\{2\}} y_{\{1, 2k + 1\}}$. \\
			\end{itemize}
		Hence $K[\Qc_{\overline{C_{2k + 1}}}]$ is quadratic. 
		Therefore $K[\Qc_{\overline{C_{2k + 1}}}]$ is quadratic Cohen-Macaulay. 
		\item 
		By (1), $K[\Qc_{\overline{C_{2k + 1}}}] \cong K[Y] / I_{\Qc_{\overline{C_{2k + 1}}}}$ is Cohen-Macaulay with 
		$\dim K[\Qc_{\overline{C_{2k + 1}}}] = 2k + 2$. 
		We note that 
		${\bf y} = y_{\emptyset}, y_{\{1\}} - y_{\{2, 3\}}, y_{\{2\}} - y_{\{3, 4\}},$ 
		$\ldots, y_{\{2k - 1\}} - y_{\{2k, 2k + 1\}},$ $y_{\{2k\}} - y_{\{1, 2k + 1\}}, y_{\{2k + 1\}} - y_{\{1, 2\}}$ 
		is a regular sequence of $K[Y] / I_{\Qc_{\overline{C_{2k + 1}}}}$. 
		Then we have that 
		\[
		\frac{K[Y]}{I_{\Qc_{\overline{C_{2k + 1}}}} + ({\bf y})} \cong \frac{K[y_{\{1\}}, y_{\{2\}}, \ldots, y_{\{2k + 1\}}]}{I_{2k + 1}}  
 		\]
		is a artinian quadratic Cohen-Macaulay ring, where $I_{2k + 1}$ is generated by the followings: \\
		\begin{itemize}
			\item $y_{\{i\}} y_{\{i + 1\}}$ \ \ \  $(1 \le i \le 2k)$;
			\item $y_{\{1\}} y_{\{2k + 1\}}$; 
			\item $y_{\{i\}}^2 - y_{\{i - 1\}} y_{\{i + 2\}}$\ \ \  $(2 \le i \le 2k - 1)$;
			\item $y_{\{1\}}^2 - y_{\{3\}} y_{\{2k + 1\}}, y_{\{2k\}}^2 - y_{\{1\}} y_{\{2k - 1\}}$, $y_{\{2k + 1\}}^2 - y_{\{2\}} y_{\{2k\}}$. \\
		\end{itemize}
		Assume $k \ge 4$. 
		Then both $y_{\{2k + 1\}}^2 \cdot \prod_{i = 1}^{k - 1} y_{\{2i\}}$ and 
		\begin{eqnarray*}
		\left\{ \begin{array}{ll}
		\displaystyle \prod_{i = 1}^{\frac{2k + 1}{3}} y_{\{3i\}} & (k \equiv 1 \ {\rm mod} \ 3),  \\
		\displaystyle y_{\{2k + 1\}} \cdot \prod_{i = 1}^{\frac{2k - 1}{3}} y_{\{3i\}} & (k \equiv 2 \ {\rm mod} \ 3),  \\
		\displaystyle y_{\{2k\}}^2 \cdot \prod_{i = 1}^{\frac{2k - 3}{3}} y_{\{3i\}} & (k \equiv 0 \ {\rm mod} \ 3),  \\
		\end{array}
		\right.
		\end{eqnarray*}
		$\ $
		
		\noindent are socle elements of $K[y_{\{1\}}, y_{\{2\}}, \ldots, y_{\{2k + 1\}}] / I_{2k + 1}$, 
		hence it is not Gorenstein. 
		Therefore $K[\Qc_{\overline{C_{2k + 1}}}]$ is not Gorenstein for all $k \ge 4$. 
		\item 
		By the proof of (2), we have 
		\[
		\frac{K[Y]}{I_{\Qc_{\overline{C_{7}}}} + ({\bf y})} \cong \frac{K[y_{\{1\}}, y_{\{2\}}, \ldots, y_{\{7\}}]}{I_7}. 
 		\]
		
		Let $<_{\rm rev}$ be the reverse lexicographic order on $K[y_{\{1\}}, y_{\{2\}}, \ldots, y_{\{7\}}]$ 
		induced by the ordering $y_{\{1\}} < y_{\{2\}} < \cdots < y_{\{7\}}$. 
		Then the initial ideal ${\rm in}_{<_{\rm rev}}(I_7)$ is generated by the following 
		monomials: 
		\[
		(y_{\{1\}}y_{\{2\}}, y_{\{2\}}y_{\{3\}}, y_{\{3\}}y_{\{4\}}, y_{\{4\}}y_{\{5\}}, y_{\{5\}}y_{\{6\}}, 
		y_{\{6\}}y_{\{7\}}, y_{\{1\}}y_{\{7\}},  
		\]
		\[
		y_{\{1\}}^3, y_{\{2\}}^2, y_{\{3\}}^2, y_{\{4\}}^2, 
		y_{\{5\}}^2, y_{\{6\}}^2, y_{\{7\}}^2, y_{\{3\}}y_{\{7\}}, y_{\{1\}}^2 y_{\{4\}}, y_{\{1\}}^2 y_{\{6\}}, 
		y_{\{2\}}y_{\{5\}}y_{\{7\}}). 
		\]
		From this, we can compute that the Hilbert series of $\displaystyle \frac{K[y_{\{1\}}, y_{\{2\}}, \ldots, y_{\{7\}}]}{{\rm in}_{<_{\rm rev}}(I_7)}$ is $1 + 7t + 14t^2 + 7t^3 + t^4$. 
		Hence $h(K[\Qc_{\overline{C_{7}}}]) = (1, 7, 14, 7, 1)$, therefore it is Gorenstein. 
		\item Assume that there exists a monomial order $<$ on $K[Y]$ such that 
		the Gr\"{o}bner bases of $I_{\Qc_{\overline{C_{2k + 1}}}}$ with respect to $<$ 
		is quadratic. \\
		We may assume that $y_{\{1\}}y_{\{2, 3\}} < y_{\{3\}}y_{\{1, 2\}}$. 
		Then $y_{\{3\}}y_{\{4, 5\}} < y_{\{5\}}y_{\{3, 4\}}$ since 
		$y_{\{5\}}y_{\{1, 2\}}y_{\{3, 4\}} - y_{\{1\}}y_{\{2, 3\}}y_{\{4, 5\}} \in I_{\Qc_{\overline{C_{2k + 1}}}}$ and its initial monomial is $y_{\{5\}}y_{\{1, 2\}}y_{\{3, 4\}}$. 
		Since $y_{\{7\}}y_{\{3, 4\}}y_{\{5, 6\}} - y_{\{3\}}y_{\{4, 5\}}y_{\{6, 7\}} \in I_{\Qc_{\overline{C_{2k + 1}}}}$ and its initial monomial is $y_{\{7\}}y_{\{3, 4\}}y_{\{5, 6\}}$, 
		 we have $y_{\{5\}}y_{\{6, 7\}} < y_{\{7\}}y_{\{5, 6\}}$. 
		 By repeating this argument, we have 
		 \begin{center}
		 $y_{\{1\}}y_{\{2, 3\}} < y_{\{3\}}y_{\{1, 2\}}, $
		
		 $y_{\{3\}}y_{\{4, 5\}} < y_{\{5\}}y_{\{3, 4\}}, $
		 
		 $\vdots$
		 
		 $y_{\{2k - 1\}}y_{\{2k, 2k + 1\}} < y_{\{2k + 1\}}y_{\{2k - 1, 2k\}}$, 
		 
		 $y_{\{2k + 1\}}y_{\{1, 2\}} < y_{\{2\}}y_{\{1, 2k + 1\}}$, 
		 
		 $y_{\{2\}}y_{\{3, 4\}} < y_{\{4\}}y_{\{2, 3\}}$, 
		 
		 $y_{\{4\}}y_{\{5, 6\}} < y_{\{6\}}y_{\{4, 5\}}$, 
		 
		 $\vdots$
		 
		 $y_{\{2k - 2\}}y_{\{2k - 1, 2k\}} < y_{\{2k\}}y_{\{2k - 2, 2k - 1\}}$, 
		 
		 $y_{\{2k\}}y_{\{1, 2k + 1\}} < y_{\{1\}}y_{\{2k, 2k + 1\}}$. 
		 \end{center}
		 These inequalities induce a contradiction. 
		 Hence we have the desired conclusion. 
		 \qedhere
	\end{enumerate} 
\end{proof}

$\ $

We can check that $R = K[\Qc_{\overline{C_{7}}}]$ is not Koszul to check 
$\beta_{34}^{R}(K) = 1 \neq 0$
by using Macaulay2.  
Hence we have 

\begin{Corollary}
\normalfont
The toric ring $K[\Qc_{\overline{C_{7}}}]$ is non-Koszul quadratic Gorenstein. 
\end{Corollary}

We can construct an infinite family of non-Koszul quadratic Gorenstein toric rings 
by using stable set polytopes. 

\begin{Proposition}
\normalfont
Let $k \ge 1$ be an integer. 
Let $G$ be a graph on $[2k + 7]$ such that $\overline{G} = C_7 \cup K_2 \cup \cdots \cup K_2$ 
and the labeling of vertices is as follows: 

$\ $

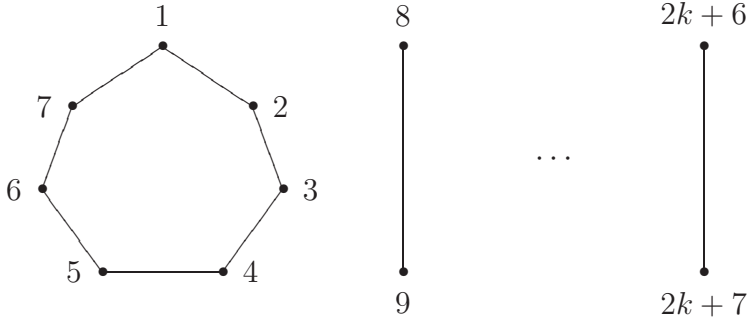
	\begin{xy}
		\ar@{} (0,  0); (20,  0) *++!R{5} *\dir<4pt>{*} = "A", 
		\ar@{-} "A"; (36,  0) *++!L{4} *\dir<4pt>{*} = "B", 
		\ar@{} "B"; (40, 22) *++!L{2} *\dir<4pt>{*} = "C", 
		\ar@{-} "C"; (28, 30) *++!D{1} *\dir<4pt>{*} = "D",
		\ar@{} "A"; (16,  22) *++!R{7} *\dir<4pt>{*} = "E",
		\ar@{-} "D"; "E"; 
		\ar@{-} "C"; (44, 11) *++!L{3} *\dir<4pt>{*} = "F",
		\ar@{-} "F"; "B"; 
		\ar@{-} "A"; (12, 11) *++!R{6} *\dir<4pt>{*} = "G",
		\ar@{-} "G";"E";
		\ar@{} (0,  0); (60,  0) *++!U{9} *\dir<4pt>{*} = "H",
		\ar@{-} "H"; (60, 30) *++!D{8} *\dir<4pt>{*} = "I";
		\ar@{} "I"; (80, 15) *{\cdots};
		\ar@{} (0,  0); (100,  0) *++!U{2k + 7} *\dir<4pt>{*} = "J",
		\ar@{-} "J"; (100, 30) *++!D{2k + 6} *\dir<4pt>{*}
	\end{xy}

\noindent Then we have

\begin{enumerate}
	\item $K[\Qc_{G}]$ is quadraic Gorenstein such that
	\[ 
	H_{K[\Qc_{G}]}(t) = (1 + 7t + 14t^2 + 7t^3 + t^4)(1 + t)^k / (1 - t)^{2k + 8}.
	\]
	\item $K[\Qc_{G}]$ is not Koszul. 
\end{enumerate}
\end{Proposition}
\begin{proof}
	\begin{enumerate}
		\item 
		By \cite[Theorem 3.2]{MOS}, we have that the toric ideal $I_{\Qc_G}$ is 
		generated by the following binomials: 
		\begin{itemize}
			\item $y_{\{i\}} y_{\{i + 1\}} - y_{\emptyset} y_{\{i, i + 1\}}$\ \ \  $(1 \le i \le 6)$;
			\item $y_{\{1\}} y_{\{7\}} - y_{\emptyset} y_{\{1, 7\}}$; 
			\item $y_{\{i\}} y_{\{i + 1, i + 2\}} - y_{\{i + 2\}} y_{\{i, i + 1\}}$\ \ \  $(1 \le i \le 5)$;
			\item $y_{\{6\}} y_{\{1, 7\}} - y_{\{1\}} y_{\{6, 7\}}$, $y_{\{7\}} y_{\{1, 2\}} - y_{\{2\}} y_{\{1, 7\}}$; 
			\item $y_{\{2i\}} y_{\{2i + 1\}} - y_{\emptyset} y_{\{2i, 2i + 1\}}$\ \ \  $(4 \le i \le k + 3)$. \\
		\end{itemize}
		Let $K[Y] = K[y_{W} \mid W \in S(G)]$. 
		Then $K[\Qc_{G}] \cong K[Y] / I_{\Qc_G}$. 
		Note that 
		${\bf y} = y_{\emptyset}, y_{\{1\}} - y_{\{2, 3\}}, y_{\{2\}} - y_{\{3, 4\}},$ 
		$\ldots, y_{\{5\}} - y_{\{6, 7\}},$ $y_{\{6\}} - y_{\{1, 7\}}, y_{\{7\}} - y_{\{1, 2\}},$ $y_{\{8\}} - y_{\{9\}}, \ldots, y_{\{2k + 6\}} - y_{\{2k + 7\}}, y_{\{8, 9\}}, \ldots, y_{\{2k + 6, 2k + 7\}}$ 
		is a regular sequence of $K[Y] / I_{\Qc_G}$. 
		Hence we have
		\[
		\frac{K[Y]}{I_{\Qc_{G}} + ({\bf y})} 
		\cong \frac{K[y_{\{1\}}, y_{\{2\}}, \ldots, y_{\{7\}}]}{I_7} \otimes_{K} \frac{K[y_{\{2i\}} \mid 4 \le i \le k + 3]}{(y_{\{2i\}}^2 \mid 4 \le i \le k + 3)}. 
 		\]
		Thus  the Hilbert series of $K[Y] / I_{\Qc_{G}} + ({\bf y})$ is 
		$(1 + 7t + 14t^2 + 7t^3 + t^4)(1 + t)^k$. 
		Therefore we have the desired conclusion.  
		\item $K[\Qc_{\overline{C_{7}}}]$ is a conbinatorial pure subring (see \cite{CPS}) 
		of $K[\Qc_{G}]$. 
		Since $K[\Qc_{\overline{C_{7}}}]$ is not Koszul, 
		hence $K[\Qc_{G}]$ is not Koszul by \cite[Proposition 1.3]{CPS}. 
		\qedhere
	\end{enumerate}
\end{proof}

\section{Questions}

As the end of this note, we present some questions. 

Recall that the $h$-vector of $K[\Qc_{\overline{C_{7}}}]$ is $(1, 7, 14, 7, 1)$. 
Hence the following question is interesting. 

\begin{Question}
Does exist a non-Koszul quadratic Gorenstein algebra $R$ such that 
$h(R) = (1, n_1, n_2, n_1, 1)$ and $n_1 \le 6$ ? 
\end{Question}
Note that, in this case $n_1 \ge 4$ since $R$ is quadratic. 
Since $n_1 = {\rm embdim}\ R - \dim R$ and 
${\rm embdim}\ K[\Qc_{G}] = \#S(G) = 1 + n + \#\{ W \in S(G) \mid \#W \ge 2 \}$ and $\dim K[\Qc_{G}] = n + 1$, if ${\rm embdim}\ K[\Qc_{G}] - \dim K[\Qc_{G}] \le 6$, then 
$\#\{ W \in S(G) \mid \#W \ge 2 \} \le 6$. 
In particular, we have $\alpha(G)  =  2$, where $\alpha(G) := \max\{ \#W \mid W \in S(G) \}$ 
is the {\em stability number} of $G$. 
Since if $G$ is perfect graph with $\alpha(G)  =  2$ then $\overline{G}$ is bipartite, 
In this case $G$ is not perfect. 

Let $G$ be a graph on $[n]$ and with $E(G)$ its edge set. 
The {\em edge ring} of $G$, denoted by  $K[G]$, is defined by
\[
K[G] := K[x_ix_j \mid \{i, j\} \in E(G)] \subset K[x_1, \ldots, x_n]. 
\]

The second question is 


\begin{Question}
Does exist a graph $G$ such that the edge ring $K[G]$ is non-Koszul 
quadratic Gorenstein ?
\end{Question}

In \cite[Theorem 1.2]{notKoszul}, a criterion for the edge ring $K[G]$ of $G$ to be quadratic 
is given. 
Moreover, in \cite{HNOS}, a class of graphs with the property that the toric ideal $I_G$ of 
the edge ring $K[G]$ of $G$ is quadratic but $I_G$ possesses no quadratic Gr\"{o}bner bases 
is studied. 
A graph $G$ is said to be $(*)$-{\em minimal} if $G$ satisfies the above property and every 
induced subgraph $H \subsetneq G$ does not satisfy the property. 
By the computation by using Macaulay2, we have that 
if $G$ is $(*)$-minimal and the edge ring $K[G]$ is non-Koszul quadratic Gorenstein, 
then $n \ge 9$. 

\bigskip

\noindent
{\bf Acknowledgment.}
The author wish to thank Professor Takayuki Hibi for his financial support.


\begin{thebibliography}{99}

\bibitem{seqKoszul}
A.~Aramova, J.~Herzog and T.~Hibi, 
Shellability of semigroup rings, 
{\em Nagoya Math. J.} {\bf 168} (2002), 65--84. 

\bibitem{IKoszul}
S.~Blum, 
Initially Koszul algebras, 
{\em Beitr\"{a}ge Algebra Geom.} {\bf 41} (2000), 455--467. 

\bibitem{CFS}
T.~Chappell, T.~Friedl and R.~Sanyal, 
Two double poset polytopes, 
arXiv:1606.04938. 

\bibitem{CRST}
M.~Chudnovsky, N.~Robertson, P.~Seymour and R.~Thomas, 
The strong perfect graph theorem, 
{\em Ann. of Math. (2)} {\bf 164} (2006), 51--229. 

\bibitem{UKoszul}
A.~Conca, 
Universally Koszul algebras, 
{\em Math. Ann.} {\bf 317} (2000), 329--346. 

\bibitem{CDR}
A.~Conca, E.~De Negri and M.~E.~Rossi, 
Koszul algebras and regularity, 
Commutative algebra, 285--315, Springer, New York, 2013.

\bibitem{CRV}
A.~Conca, M.~E.~Rossi and G.~Valla, 
Gr\"{o}bner flags and Gorenstein algebras, 
{\em Compositio Math.} {\bf 129} (2001), 95--121. 

\bibitem{D}
R.~Diestel, 
Graph Theory, Fourth edition, 
Grad. Texts Math. {\bf 173}, Springer, 2010. 

\bibitem{EN}
A.~Engstr\"{o}m and P.~Nor\'{e}n, 
Ideals of graph homomorphisms, 
{\em Ann. Comb. } {\bf 17} (2013), 71--103. 

\bibitem{Fr}
R.~Fr\"{o}berg, 
Determination of a class of Poincar\'{e} series, 
{\em Math. Scand.} {\bf 37} (1975), 29--39. 

\bibitem{Macaulay2}
D.~R.~Grayson and M.~E.~Stillman, 
Macaulay2, a software system for research in algebraic geometry, 
Available at \texttt{http://www.math.uiuc.edu/Macaulay2/}. 

\bibitem{HHR}
J.~Herzog, T.~Hibi and G.~Restuccia, 
Strongly Koszul algebras, 
{\em Math. Scand.} {\bf 86} (2000), 161--178. 

\bibitem{Hibi1987}
T. Hibi, Distributive lattices, affine semigroup rings and algebras with 
straightening laws, {\it in} ``Commutative Algebra and Combinatorics'' 
(M. Nagata and H. Matsumura, Eds.), Advanced Studies in Pure Math., 
Volume 11, North--Holland, Amsterdam, 1987, pp. 93 -- 109.

\bibitem{Chain}
T.~Hibi and N.~Li, 
Chain polytopes and algebras with straightening laws, 
{\em Acta Math. Vietnam.} {\bf 40} (2015), 447--452. 

\bibitem{HM}
T.~Hibi and K.~Matsuda, 
Quadratic Gr\"{o}bner bases of twinned order polytopes, 
{\em European J. Combin.} {\bf 54} (2016), 187--192. 

\bibitem{HMS}
T.~Hibi, K.~Matsuda and H.~Ohsugi, 
Strongly Koszul edge rings, 
{Acta Math. Vietnam.} {\bf 41} (2016), 69--76. 

\bibitem{HMOS}
T.~Hibi, K.~Matsuda, H.~Ohsugi and K.~Shibata,
Centrally symmetric configurations of order polytopes,
{\em J. Algebra} {\bf 443} (2015), 469--478. 

\bibitem{GF}
T.~Hibi, K.~Matsuda and A.~Tsuchiya, 
Gorenstein Fano polytopes arising from order polytopes and chain polytopes, 
arXiv:1507.03221. 

\bibitem{HMT}
T.~Hibi, K.~Matsuda and A.~Tsuchiya, 
Quadratic Gr\"{o}bner bases arising from partially ordered sets, 
{\em Math. Scand.}, to appear. 

\bibitem{HNOS}
T.~Hibi, K.~Nishiyama, H.~Ohsugi and A.~Shikama, 
Many toric ideals generated by quadratic binomilals possess no 
quadratic Gr\"{o}bner bases, 
{\em J. Algebra} {\bf 408} (2014), 138--146. 

\bibitem{Hochster}
M.~Hochster, 
Rings of invariants of tori, Cohen-Macaulay rings generated by monomials, 
and polytopes, 
{\em Ann. of Math. (2)} {\bf 96} (1972), 318--337. 

\bibitem{Mancini}
F.~Mancini, 
Graph modification problems related to graph classes, 
Dissertation, University of Bergen, 2008. 


\bibitem{Matsuda}
K.~Matsuda, 
Strong Koszulness of toric rings associated with stable set polytopes of trivially
perfect graphs, 
{\em J. Algebra Appl.} {\bf 13} (2014), 1350138 [11 pages]. 

\bibitem{MO}
K.~Matsuda and H.~Ohsugi, 
Reverse lexicographic Gr\"{o}bner bases and strongly Koszul toric rings, 
{\em Math. Scand.}, to appear. 

\bibitem{MOS}
K.~Matsuda, H.~Ohsugi and K.~Shibata, 
Toric rings and ideals of stable set polytopes, 
arXiv:1603.01850. 

\bibitem{CPS}
H.~Ohsugi, J.~Herzog and T.~Hibi, 
Combinatorial pure subrings, 
{\em Osaka J. Math.} {\bf 37} (2000), 745--757. 

\bibitem{notKoszul}
H.~Ohsugi and T.~Hibi, 
Toric ideals generated by quadratic binomials, 
{\em J. Algebra} {\bf 218} (1999), 509--527. 

\bibitem{Gorenstein}
H.~Ohsugi and T.~Hibi, 
Special simplices and Gorenstein toric rings, 
{\em J. Combin. Theory Ser. A} {\bf 113} (2006), 718--725. 



\bibitem{Priddy}
S.~B.~Priddy, 
Koszul resolutions, 
{\em Trans. Amer. Math. Soc.} {\bf 152} (1970), 39--60. 

\bibitem{RS}
J.-E. Roos and B.~Sturmfels, 
A toric ring with irrational Poincar\'{e}-Betti series, 
{\em C. R. Acad. Sci. Paris S\"{e}r. I Math.} {\bf 326} (1998), 141--146. 

\bibitem{Shibata}
K.~Shibata, 
Strong Koszulness of the toric ring associated to a cut ideal, 
{\em Comment. Math. Univ. St. Pauli} {\bf 64} (2015), 71--80. 


\bibitem{Stanley}
R.~P.~Stanley,
Hilbert functions of graded algebras, 
{\em Adv. Math.} {\bf 28} (1978), 57--83.

\bibitem{Sturmfels}
B.~Sturmfels, 
Gr\"{o}bner Bases and Convex Polytopes, 
Amer. Math. Soc., Providence, RI, 1996. 

\bibitem{Four}
B.~Sturmfels, 
Four counterexamples in combinatorial algebraic geometry, 
{\em J. Algebra} {\bf 230} (2000), 282--294. 

 \bibitem{Tate}
J.~Tate, 
Homology of local and Noetherian rings, 
{\em Illinois J. Math.} {\bf 1} (1957), 14--27.   
 
\end{thebibliography}
\end{document}